\documentclass[leqno,12pt]{amsart}
\usepackage{amsfonts, amsmath, amssymb, lmodern, eurosym}
\usepackage[pdftex]{graphicx}
\usepackage[latin1]{inputenc}
\usepackage[english]{babel}
\usepackage[normalem]{ulem}
\usepackage{lettrine}
\usepackage{color}
\usepackage{lscape}
\usepackage[T1]{fontenc}
\usepackage[latin1]{inputenc}
\usepackage{tabularx}
\usepackage{array}
\usepackage{multirow}
\usepackage[Glenn]{fncychap}
\usepackage[english]{minitoc}
\usepackage[english]{algorithm2e}
\usepackage{listings}
\usepackage{verbatim}
\definecolor{blueumi}{rgb}{0.00,0.48,0.62}
\usepackage{tikz}
\usepackage{fancybox}
\usepackage{colortbl}
\usepackage{fancyhdr}
\usepackage{lastpage}

\newtheorem{definition}{Definition}[section]
\newtheorem{theo}{Theorem}[section]
\newtheorem{prop}[theo]{Proposition}

\newtheorem{cor}{Corollary}

\usepackage{amsmath}
\usepackage{amssymb}
\usepackage{graphicx}
\usepackage{tikz-cd}
\usepackage[T1]{fontenc}
\usepackage{times}

\begin{document}
\pagestyle{plain}

\cfoot[C]{\thepage}
\vspace*{-1cm}

\title{On Topological Complexity of $(r,\rho(R))$-mild spaces}
\author[1]{Smail BENZAKI}
\author[2]{Youssef RAMI}

{\let\thefootnote\relax\footnote{{\it Address}: Moulay Ismaïl University, Department of Mathematics  B. P. 11 201 Zitoune,  Meknès, Morocco}\let\thefootnote\relax\footnote{{\it Emails}: smail.benzaki@edu.umi.ac.ma  and y.rami@umi.ac.ma}}

 \keywords{Rational topological complexity, Commutative models, $r$-mild algebra, $r$-mild CW-complex}
 \subjclass{Primary 55P62; Secondary 55M30 }

\renewcommand{\abstractname}{Abstract}
\begin{abstract}
In this paper, we first prove the existence of relative free models of morphisms (resp. relative commutative models) in the category of $DGA(R)$ (resp. $CDGA(R)$), where $R$ is a principal ideal domain containing $\frac{1}{2}$. Next, we restrict to the category of 
 $(r,\rho(R))$-H-mild algebras and we introduce, following Carrasquel's characterization, $secat(-, R)$, the sectional category for surjective morphisms. We then apply this to the $n$-fold product of the commutative model of an $(r,\rho(R))$-mild CW-complex of finite type to introduce $TC_n(X,R)$, $mTC_n(X,R)$ and $HTC_n(X,R)$ which extend well known rational topological complexities. We do the same for $\operatorname{sc(-, \mathbb{Q})}$ to introduce analogous algebraic $\operatorname{sc(-,R)}$ in terms of their commutative models over  $R$ and prove that it is an upper bound for $secat(-, R)$. This also yields, for any   $(r,\rho(R))$-mild CW-complex,  the algebraic $tc_n(X,R)$, $mtc_n(X,R)$ and $Htc_n(X,R)$ whose relation to the homology nilpotency is investigated. In the last section,  in the same spirit, we introduce in  $DGA(R)$,  $secat(-, R)$, $\operatorname{sc(-,R)}$ and their topological correspondents. We then  prove, in particular, that $ATC_n(X,R)\leq TC_n(X,R)$ and $Atc_n(X,R)\leq tc_n(X,R)$.
\end{abstract}

\maketitle

\section{Introduction}

Rational topological complexity provides a foundation for comprehending the complexities of continuous motion planning in high-dimensional environments. Based on the notion of topological complexity, which evaluates the minimal difficulty of a motion planning algorithm on a topological space, rational topological complexity broadens this concept by including rational methods.

The {\it (higher) topological complexity} $TC_n(X)$ of a topological space $X$ is the sectional category $secat(\Delta_n)$ of the diagonal map $\Delta_n : X\rightarrow X^n$ (\cite{Fa} for $n=2$ and \cite{Ru} for $n\geq 2$). Recall that 
 sectional category $secat(f)$, called also the {\it Schwartz genus}, of a map $f : X\rightarrow Y$ is defined in \cite{Sch}, as the least $m$ such that there exist $m+1$ local homotopy sections for $f$ whose domains form an open cover of $Y$.

In this paper, we extend the scope of rational topological complexity by considering a principal ideal domain, thereby generalizing its applicability to a wider array of mathematical structures.
 This extension not only enhances our theoretical comprehension of topological complexity but also widens its potential applicability to address motion planning challenges in scenarios where rational constraints may not fully capture underlying geometric or algebraic intricacies.

To every simply connected CW-complex of finite type $X$, it is associated a (Sullivan) minimal model $(\Lambda V, d)$, which is a commutative differential graded algebra over $\mathbb{Q}$, where $\Lambda V$ denotes the free graded commutative algebra on a graded vector space $V$ and
where $d(V)\subset \Lambda^{\geqslant 2}V$. Furthermore, every cdga morphism $\varphi : (A, d)\rightarrow (B, d)$ admits a minimal relative model
$$\begin{tikzcd}
(A,d)\arrow[r,"\varphi"]\arrow[rd,hook,"i"']  &  (B,d)  \\
  &  (A\otimes\Lambda V,d)\arrow[u,"\simeq","\psi"'],
\end{tikzcd}$$
such that $\psi\circ i=\varphi$ and $\psi$ is a quasi-isomorphism \cite{Su}.

Let $R$ be a principal ideal domain containing $\frac{1}{2}$ and denote by $\rho(R)$ the least  non invertible   prime ( or $\infty$) in $R$.

When $R$ is a sub-ring of $\mathbb{Q}$,   D. Anick defined   in \cite{A} the sub-category $CW_r^{r\rho(R)}(R)$ of finite type $r$-connected CW-complexes $X$ ($r \geq 1$) satisfying  $\dim(X)\leq r\rho(R)$, such CW-complexes are called $(r,\rho(R))$-mild. He then showed  that to any space in $CW_r^{r\rho(R)}(R)$ it is associated an appropriate differential graded  Lie algebra (henceforth "dgL") $(L=\{L_i\}_{i\geq 1}, \partial)$, where each $L_i$ is $R$-free on a finite basis,  together with a differential graded algebra equivalence
\begin{equation*}\label{Phi}
\Phi: UL\stackrel{\simeq}{\longrightarrow} C_*(\Omega(X); R)
\end{equation*}
 from the enveloping algebra on $L$ to the $R$-chains on the loop space of $X$. Moreover, $\Phi$    preserves the diagonal up to a dga homotopy
 and  it is uniquely determined up to a unique dgL homotopy class
 of dgL isomorphisms.

Further, in \cite{Ha},  S. Halperin  extended Anick's result to any $R$ \cite[p. 274]{Ha} and showed that the Cartan-Eilenberg-Chevally complex $C^*(L)$ of $L$
 is an augmented commutative differential graded algebra (cdga for short) which is equivalent to $C^*(X; R)$. Indeed these are linked by the quasi-isomorphisms (morphisms inducing  isomorphisms in cohomology):
$$C^*(L) \stackrel{\simeq}{\longleftarrow} B(C_*(\Omega (X))^{\vee}\stackrel{\simeq}{\longrightarrow} C^*(X; R)$$ where  $B(C_*(\Omega (X))^{\vee}$ is the dual of the bar-construction on the algebra $C_*(\Omega (X))$.

{\it Henceforth,  we will assume that  $H_*(\Omega X, R)$ is $R$-torsion free}.  Therefore $C^*(L)$ has a  {\it commutative decomposable model}  \cite[\S 6, Theorem 10.1]{Ha}:
 \begin{equation}\label{mmodel}
 (\Lambda W, d) \longrightarrow C^*(L).
 \end{equation}
in the sense that  $W= \{W^i\}_{i\geq 2}$, with $W^i$ $R$-free on a finite basis, and  $d: W\rightarrow \Lambda ^{\geq 2}W$.
  
We rely on the work of J. M. Thiercelin-Panais, whose only reference is \cite{Th-Pa} which unfortunately contains mistakes in the proof of Propositions $1$ and $5$, which we use as a starting point to state our results. To our knowledge, there is no other corrected version, and for the sake of completeness, we prove them in section \ref{models} (cf. Propositions $2.1$ and $2.3$ below). We also establish, Propositions $2.2$ and $2.4$ below) the Propositions $2$ and $6$ in \cite{Th-Pa} (which are not explicitly proved). 

To ensure the existence of minimal models, as in \cite[\S 7]{Ha}, {\it we will frequently consider the following condition} on a graded module $H$ for some $r\geq 1$:
\begin{equation}\label{finitegenfree}
H=\left\lbrace  H^i\right\rbrace_{i\geq 0} \text{\it is finite type, } H^0 = R,\;  H^{1\leq i\leq r} = 0 \; \text{\it and } H^{r+1}\; \text{\it is } R\text{\it -free}.
\end{equation}
 Let then $\varphi : (A,d)\rightarrow (B,d)$ be a morphism   of cdga's over a principal ideal domain $R$ containing $\frac{1}{2}$. 
 
 In Proposition $2.3$ below, assuming (for $r=1$)  that $H(B,d)$ satisfy (\ref{finitegenfree}) and, moreover,  $H^2(\varphi): H^2(A,d)\rightarrow H^2(B,d)$ is injective, we will prove that.
 Such $\varphi$ might be factored  as follows:
\begin{equation}\label{mm}
\begin{tikzcd}
(A,d) \arrow[r,"\varphi"] \arrow[rd,hook,"\iota"'] & (B,d)\\
 & (A\otimes \Lambda V,d), \arrow[u,"\simeq","\Phi"']
\end{tikzcd}
\end{equation}
with $(\Lambda V,d)$ a free finite-type cdga and  the vertical arrow $\Phi$  is a  quasi-isomorphism of cdga. 
 
We then borrow the following definition from  \cite[Definition p.427]{A}, with some modification.
\begin{definition}\label{R-milddef}
A cdga over $R$ is said to be $(r,\rho(R))$-H-mild if its homology is concentrated in the range of degrees $r+1$ through $r\rho (R)$, inclusive. 
Here $\rho(R)$ is the smallest prime number in $R$.
These form a category, that we will denote going forward,  $\text{CDGA}_r(R)$
 consisting of all $(r,\rho(R))$-H-mild objects.
\end{definition}
In Proposition $2.4$ below, we will prove  (once again for $r=1$) the lifting property for any cdga morphism $\varphi : (A,d)\rightarrow (B,d)$ satisfying the same hypothesis as in Proposition $2.3$. 

As a consequence, combining  Proposition $2.3$ and Proposition $2.4$ with 
\cite[Proposition 7.7 and Remark 7.8]{Ha} we acquire that  every $(r,\rho(R))$-{H-mild} cdga $(B,d)$ has a unique up to quasi-isomorphism minimal commutative model {\it provided that $H^{r+1}(B,d)$ is torsion free} (cf. Corollary $1$ below). 
 This generalizes  Sullivan's theory \cite{Su} for cdga's over the rationals to the category $\text{CDGA}_r(R)$.

Especially, the  commutative decomposable model $(\Lambda V,d)$ of  $C^*(L)$ given in (\ref{mmodel})
is unique up to quasi-isomorphisms and depends functorially on $X$ \cite[\S 6 and Theorem 10.1]{Ha}. 
This also extends  Sullivan's original minimality established for rational spaces  in \cite{Su}.

 Let us then fix $\varphi :(A,d)\rightarrow (B,d)$ a surjective morphism of $(r, \rho(R'))$-{H-mild} cdga's, that is, $\varphi$ is in $\text{CDGA}_r(R')$:
 
  To transform the projection:
$$p_{m}:\left({A}^{\otimes m+1}, d\right) \longrightarrow\left(\frac{{A}^{\otimes m+1}}{\left(\operatorname{ker} {{(\varphi)}}\right)^{\otimes m+1}}, \overline{d} \right)$$
into a morphism of some $(r,\rho(R))$-H-mild cdga, we take $R$ with $\rho(R)$ satisfying  (cf. the beginning of  Section $3$ for more details):
\begin{equation*}\label{K'}
m\rho(R')\leq \rho(R).
\end{equation*}
 Given such $R$,  (\ref{mm}) induces the following commutative diagram:
 \begin{equation*}
\begin{tikzcd}
\left(A^{\otimes m+1}, d\right) \arrow[r,"p_{m}"] \arrow[rd,hook,"i_m"']
&  \left(\frac{{A}^{\otimes m+1}}{\left(\operatorname{ker} {{(\varphi)}}\right)^{\otimes m+1}}, \overline{d} \right) \\
  &  (A^{\otimes m+1}\otimes \Lambda W_{(m)},d). \arrow[u,"\simeq"']
\end{tikzcd}
 \end{equation*}
 Thus, introducing  $\bar{\mu} _{m+1} : A^{\otimes m+1}\otimes \Lambda W_{(m)}\stackrel{\mu_{m+1}\otimes id_{\Lambda W_{(m)}}}{\longrightarrow} A\otimes \Lambda W_{(m)}$,
 where  $\mu_{A,n} : A^{\otimes n}\rightarrow A$ is the $n$-fold product of $A$,
we then obtain the following diagram:
$$\begin{tikzcd}
\left(A^{\otimes m+1}, d\right) \arrow[r,"i_m",hook] \arrow[d,"\mu_{m+1}"'] & (A^{\otimes m+1}\otimes \Lambda W_{(m)},d) \arrow[d,"\bar{\mu}_{m+1}"] \\
A \arrow[r,"j_m"',hook]  & (A\otimes \Lambda W_{(m)},\bar{d})
\end{tikzcd}$$
by which we extend the definition  $secat(\varphi)$ for  any surjective morphism $\varphi :(A,d)\rightarrow (B,d)$  in  $\text{CDGA}_r(R)$ satisfying (\ref{finitegenfree})
 as follows: 
\begin{equation}\label{secat}
 secat(\varphi) := \min\{m\; | \; \exists \; r_m: (A^{\otimes m+1}\otimes \Lambda W_{(m)},d)\rightarrow (A,d) \;  \hbox{ s. t.}\;  r_m\circ \iota_m = \mu_{m+1}\}.
\end{equation}
Notice that $r_m: (A^{\otimes m+1}\otimes \Lambda W_{(m)},d)\rightarrow (A,d)$ is a cdga morphism.
 
Let $(\Lambda V,d)\longrightarrow C^*(L)$ be a decomposable commutative model of a finite-type $(r,\rho(R))$-mild CW-complex  $X$. 
 Applying the above process to the $n$-fold product $\mu_{\Lambda V,n}:(\Lambda V)^{\otimes n}\rightarrow \Lambda V$ yields an extension of the notion of higher topological complexity to the category $CW_r^{r\rho(R)}(R)$ (recall that we assumed that $H_*(\Omega X, R)$ is torsion free): 
\begin{equation*}\label{TC-n}
  TC_n(X,R) := secat(\mu_{\Lambda V,n}).
\end{equation*}

Next we adapt the "non homotopy" invariant $sc$ introduced by Carrasquel in \cite{JCar}, to the case of principal ideal domains containing $\frac{1}{2}$. Let then  $\varphi :(A,d)\rightarrow (B,d)$ be a surjective morphism in  $\text{CDGA}_r(R)$ satisfying (\ref{finitegenfree}), we have a unique commutative model of the projection $\Gamma_m : (A,d)\rightarrow (\frac{A}{(\ker(\varphi))^{m+1}}, \bar{d})$:
$$\begin{tikzcd}
\left(A, d\right) \arrow[r,"\Gamma_{m}"] \arrow[rd,hook,"\iota_m"']
&  \left(\frac{A}{\left(\operatorname{ker} {{(\varphi)}}\right)^{m+1}}, \overline{d} \right) \\
  &  (A\otimes \Lambda W_{(m)},d), \arrow[u,"\simeq"']
\end{tikzcd}$$
across which we introduce the following
\begin{equation*}
  sc(\varphi) :=  \min\{m\; | \; \exists \; k_{m}:  (A\otimes \Lambda W_{(m)},d)\rightarrow (A,d) \; \hbox{a cdga morphism s. t.}\;  k_m\circ \iota_m = id_{(A,d)}\}.
\end{equation*}
Similarly, applying this definition to the $n$-fold product  $\mu_{\Lambda V,n}:(\Lambda V)^{\otimes n}\rightarrow \Lambda V$ associated to a space $X$ in $CW_r^{r\rho(R)}(R)$, such that $H_*(\Omega X, R)$ is torsion free, yields
\begin{equation*}
  tc_n(X,R) := sc(\mu_{\Lambda V,n}).
\end{equation*}
 
 Once again, following \cite{JCar} (see also \cite{JCar2}) we will give, in  $\text{DGA}_r(R)$, an equivalent definition to $secat(\varphi)$ and use it  to introduce, in Definition $3.1$, notions of $msecat$ and $Hsecat$. We also introduce their corresponding $msc$  and $Hsc$ in Definitions $4.1$.
  
The main result of this paper reads
 \begin{theo}
\begin{itemize}
\item[(i)] $secat(\varphi)\leq sc(\varphi)$,
\item[(ii)] $msecat(\varphi)\leq msc(\varphi)$,
\item[(iii)] $Hsecat(\varphi)\leq Hsc(\varphi)$.
\end{itemize}
\end{theo}

Next we prove that for any surjective morphism we have
\begin{equation*}
nil \ker(H(\varphi))\leq secat(\varphi)\leq sc(\varphi) \leq Hnil \ker(\varphi) \leq nil \ker(\varphi),
\end{equation*}
 where $Hnil \ker(\varphi)$ is the homology nilpotency of $\ker(\varphi)$, that is, the smallest integer  $k$ for which $(\ker(\varphi))^{k+1}$ is contained in some acyclic ideal of $A$. As a corollary we have
\begin{equation*}
nil \ker(H(\mu_{\Lambda V,n}))\leq TC_n(X,R)\leq tc_n(X,R)\leq Hnil \ker(\mu_{\Lambda V,n})\leq nil \ker(\mu_{\Lambda V,n}).
\end{equation*}

The rest of the paper is structured as follows. The first section is dedicated to providing basic notions used throughout the paper while the second section is for proving the aforementioned Propositions $2.1$ to $2.3$, as well as their corresponding  lifting lemmas. 
In section three, we extend, to the category $\text{CDGA}_r(R)$, the notion of  sectional category of a surjective morphism in terms of commutative models and establish some inequalities. The fourth section generalizes the notion of $\operatorname{sc}$, introduced in \cite{JCar} in rational homotopy context,   to the category  $\text{CDGA}_r(R)$. We then  prove that it upper bounds sectional category. We recall in Section $5$, the notion of homology nilpotency and establish some lower and upper bounds of both $\operatorname{secat}$ and $\operatorname{sc}$. Lastly, we introduce in Section $6$ the notions of $\operatorname{Asecat}$ and $\operatorname{Asc}$  in the context of differential graded algebras over $R$ and prove that, in $\text{DGA}_r(R)$, their corrsponding $ATC(X,R)$ and $Atc(X,R)$ lowers respectively $TC(X,R)$ and $tc(X,R)$.

\section{Preliminaries}
In this section, we provide a brief reminder of essential notions used in this work.

A differential graded algebra $A$ (dga) is a graded algebra together with a differential in $A$ that is a derivation, in addition, if $aa'=(-1)^{|a||a'|}a'a$, $A$ is called a commutative differential graded algebra (cdga). A morphism of differential graded algebras $f : (A,d)\rightarrow (B,d)$ is a morphism of graded algebras compatible with the grading and $d$.

The free product of two algebras $A$ and $B$ is the coproduct in the category of algebras, namely, $A\sqcup B=T(A,B)/I$ where $T(A,B)=\bigoplus_n^\infty T_n$ where \cite[Remark p. 232]{BL}:
$$T_0=R,\quad T_1=A\oplus B,\quad T_2=(A\otimes A)\oplus (A\otimes B)\oplus (B\otimes A)\oplus (B\otimes B),\cdots \cdots $$
and $I$ is the two-sided ideal generated by elements of the form
$$a\otimes a'-aa',\quad b\otimes b'-bb', \quad 1_A-1_B.$$
The differential on $T(A,B)$ is that induced by those of $A$ and $B$. Clearly, it passes to the quotient modulo $I$ so that we obtain the dga structure on $A\sqcup B$.

In the category of commutative algebras, the coproduct becomes the tensor product.

A (left) module over a differential graded algebra $(A, d)$ is an $A$-module $M$, together with a differential $d$ in $M$ satisfying
$d(a\cdot m)=da\cdot m+(-1)^{|a|}a\cdot dm, \hspace*{0.2cm}m\in M,\hspace*{0.1cm}a\in A.$
A morphism of (left) modules over a dga $(A,d)$, is a morphism $f:(M,d)\rightarrow (N,d)$ of graded $A$-modules satisfying $df =(-1)^{|f|} fd$.

A quasi-isomorphism is a morphism inducing an isomorphism in homology.

A homotopy retraction of cdga (resp. $(A, d)$-module) for a cdga morphism $\psi : (A, d)\rightarrow (B, d)$ is a cdga (resp. $(A, d)$-module) morphism $r : (A\otimes\Lambda V,d)\rightarrow (A, d)$ such that $r\circ i = Id_A$ where $i : (A, d)\rightarrow (A\otimes\Lambda V,d)$ is a model of $\psi$.

\section{Free and commutative models}\label{models}

Let $R$ be a principal ideal domain. In this section, we prove the existence of a free model of a dga morphism or $R$ modules under certain hypothesis, and then we prove that the lifting property is guaranteed under the hypothesis that the quasi-isomorphism is surjective. Next we pass to the category of cdga's and we establish analogous results.

\begin{prop}\label{free case} Let $\left(A, d_A\right)$ and $\left(B, d_B\right)$ be two $R$-differential graded algebras such that
\begin{itemize}
\item[(i)] $H^0\left(A, d_A\right)=H^0\left(B, d_B\right)=R$ and $H^1\left(A, d_A\right)=H^1\left(B, d_B\right)=0$.
\item[(ii)] $H^i\left(A, d_A\right)$ and $H^i\left(B, d_B\right)$ are finitely generated $R$-modules for all $i$.
\item[(iii)] $H^2\left(B, d_B\right)$ is a free $R$-module.
\item[(iv)] $A$ is $R$-free.
\end{itemize}
Let $f: (A,d_A) \rightarrow (B,d_B)$ be a morphism of $R$-dga's such that
\begin{itemize}
\item[(v)] $H^2(f): H^2(A,d_A) \rightarrow H^2(B,d_B)$ is injective.
\end{itemize}
Then there exists a commutative diagram in the category of $R$-dga's of the form:
$$
\begin{tikzcd}
(A, d_A) \arrow[r,"f"]\arrow[rd,hook,"i"'] & (B, d_B)  \\
 & (A \sqcup TV, d) \arrow[u,"\varphi"'] ,
\end{tikzcd}
$$
where:
\begin{itemize}
\item[1)] $V$ is a free $R$-module and $TV$ is the tensor algebra over $V$.
\item[2)] $A\sqcup T(V)$ designate the free product of $A$ and $TV$, and $i$ is the inclusion ($i(a)=a$).
\item[3)] $\varphi$ is a quasi-isomorphism of $R$-dga's.
\end{itemize}
\end{prop}

\begin{proof}
We proceed by induction $k\geqslant 0$.

$\bullet k=0$. We put $V_0^0=R$ then $TV_0^0\cong R$ and $A\sqcup TV_0^0\cong A$. Then $\varphi_0=f$.

$\bullet k=1$. Since $H^1(A,d_A)=H^1(B,d_B)=0$, we take $V_1^1$ to be $0$. Thus $\varphi_1=f$.

$\bullet k=2$. By hypothesis, we have $H^2(f):H^2(A,d_A)\rightarrow H^2(B,d_B)$ is injective. We write $H^2(B,d_B)=\operatorname{Im}(H^2(f))\oplus \mathcal{B}_2^2$. Since $R$ is a principal ideal domain and $H^2(B,d_B)$ is finitely generated, then so is $\mathcal{B}_2^2$.

Denote by $[\beta_2^1]$, $\cdots$, $[\beta_2^{r_2}]$ a basis of $\mathcal{B}_2^2$, so $\beta_2^i$ is a cocycle, $d\beta_2^i=0$. We now introduce $$V_2^2=\langle \alpha_2^1,\cdots ,\alpha_2^{r_2} \rangle ,$$ and we extend $d_A$ to $A\sqcup TV_2^2$ by putting $d\alpha_2^i=0$ and $\varphi_2(\alpha_2^i)=\beta_2^i$ for $1\leqslant i\leqslant r_2$.
Therefore $$\varphi_2:(A\sqcup TV_2^2,d)\rightarrow (B,d_B)$$ is a morphism satisfying $d_B\circ\varphi_2=\varphi_2\circ d$ and extending $f$. Clearly $H^2(\varphi_2):H^2(A\sqcup TV_2^2,d)\rightarrow H^2(B,d_B)$ is surjective. The classes that derive from this step are exactly $[\alpha_2^i]$, $1\leqslant i\leqslant r_2$, such that $H^2(\varphi_2)([\alpha_2^i])=[\beta_2^i]\neq 0$, that is, $H^2(\varphi_2)$ remains injective. Therefore $H^2(\varphi_2)$ is an isomorphism.

If $H^3(\varphi_2)$ is not surjective, we progress similarly to the previous case and acquire $$V_2^3=\langle {\alpha'}_2^1,\cdots ,{\alpha'}_2^{{r'}_2} \rangle $$ and we extend $\varphi_2$ to what we denote, for the time being,  $$\psi_2:(A\sqcup T(V_2^2\oplus V_2^3),d)\rightarrow (B,d_B)$$ such that $H^3(\psi_2)$ is surjective.

So far, we have constructed $\psi_2:(A\sqcup T(V_2^2\oplus V_2^3),d)\rightarrow (B,d_B)$ such that:
\begin{itemize}
\item[($a_2$)] $V=V_2=V_2^2\oplus V_2^3$ is a free $R$-module.
\item[($b_2$)] $H^{\leqslant 2}(\psi_2)$ is an isomorphism.
\item[($c_2$)] $H^3(\psi_2)$ is surjective.
\end{itemize}
Notice that the subscripts represent the step and not the degree.

Now suppose that we have constructed $\psi_{k-1}:(A\sqcup T(V_{\leqslant k-1}),d)\rightarrow (B,d_B)$ satisfying:
\begin{itemize}
\item[($a_{k-1}$)] $V=V_{\leqslant k-1}$ is a free $R$-module.
\item[($b_{k-1}$)] $H^{\leqslant k-1}(\psi_{k-1})$ is an isomorphism.
\item[($c_{k-1}$)] $H^k(\psi_{k-1})$ is surjective.
\end{itemize}
So that now we are in the $k$-th step. We know that $H^k(A\sqcup T(V_{\leqslant k-1}),d)$ is finitely generated since $H^k(A,d_A)$ is finitely generated and the number of generating classes, of degree less than $k-1$, following from $TV_{\leqslant k-1}$ are finite. Again since $R$ is a principal ideal domain, then $\ker (H^k(\psi_{k-1}))$ is also finitely generated.

Using the classifying theorem of finitely generated $R$-modules, we choose $[z_k^1]$, $\cdots$, $[z_k^{n_k}]$ to be a basis of $\ker (H^k(\psi_{k-1}))$ such that for $i\leqslant m_k\leqslant n_k$, $[z_k^i]$ has finite order and for $m_k<i\leqslant n_k$, $[z_k^1]$ has infinite order. Consequently, there exist $\lambda_k^1, \cdots , \lambda_k^{m_k}\in R$, satisfying $x_k^i=\lambda_k^i z_k^i$ is a coboundary for $1\leqslant i\leqslant m_k$. As $[\psi_{k-1}(z_k^i)]=0$, there is then  $b_k^1, \cdots ,b_k^{n_k}$ in $B$ such that 
$$\psi_{k-1}(z_k^i)=d_Bb_k^i,\; 1\leqslant i\leqslant n_k.$$
Now we introduce $$V_k^{k-1}=u_k^1R\oplus\cdots\oplus u_k^{n_k}R$$ with  $|u_k^i|=k-1$, and extend $d$ on $A\sqcup T(V_{\leqslant k-1}\oplus V_k^{k-1})$ by putting, for $1\leqslant i\leqslant n_k$, 
$$du_k^i=z_k^i,\; 1\leqslant i\leqslant n_k.$$
Next, we extend $\psi_{k-1}$ by defining the image of each $u_k^i$, for the time being, we keep noting the morphism by $\psi_{k-1}$. We should have $$d_B\circ \psi_{k-1}(u_k^i)= \psi_{k-1}(du_k^i)=\psi_{k-1}(z_k^i) =d_Bb_k^i.$$ 
This creates new non-zero classes $[\psi_{k-1}(u_k^i)-b_k^i]$ in $H^{k-1}(B,d_B)$, some  $1\leqslant i\leqslant {n'_k}$.
To insure that $H^{k-1}(\psi_{k-1})$ is still surjective, we introduce a new $R$-module 
$$W_k^{k-1}={u'}_k^1R\oplus\cdots\oplus {u'}_k^{n'_k}R$$
 such that $d{u'}_k^i=0$ and $\psi_{k-1}({u'}_k^i)=\psi_{k-1}(u_k^i)-b_k^i$, $1\leqslant i\leqslant {n'_k}$.
We then put: $$\psi_{k-1}({u}_k^i)=\psi_{k-1}({u'}_k^i)+b_k^i$$ 

Let us now deal with the impact of coboundaries $x_k^i$. If $y_k^1$, $\cdots$, $y_k^{m_k}$ in $\left( A\sqcup T(V_{\leqslant k-1}) \right)^{k-1}$ are such that $dy_k^i=x_k^i$,  then  $x_k^i=\lambda_k^i z_k^i=\lambda_k^i du_k^i=d(\lambda_k^i u_k^i)$ implies   $d(y_k^i-\lambda_k^iu_k^i)=0$. This induces new classes: $$[y_k^i-\lambda_k^iu_k^i]\in H^{k-1}(A\sqcup T(V_{\leqslant k-1}\oplus V_k^{k-1}\oplus W_k^{k-1}),d),\; 1\leqslant i\leqslant m_k.$$
Now, some of these cohomology classes may be in $\ker (H^{k-1}(\psi_{k-1}))$, resulting in $H^{k-1}(\psi_{k-1})$ not being injective. In order to remedy this, we introduce certain number (${m'}_k\leqslant m_k$) of elements $w_k^i$ with $|w_k^i|=k-2$, forming $$V_k^{k-2}=w_k^1R\oplus \cdots\oplus w_k^{{m'}_k}R$$ and we extend $d$ by setting
\begin{equation}\label{indecomposability}
dw_k^i=y_k^i-\lambda_k^iu_k^i \; \hbox{for} \; 1\leq i\leq {m'}_k.
\end{equation} 
To define $\psi_{k-1}(w_k^i)$, we refer back to the equation
$$ d_B(\psi_{k-1}(w_k^i))=\psi_{k-1}(dw_k^i)=\psi_{k-1}(y_k^i-\lambda_k^iu_k^i)$$
which involves  $c_k^i\in B$ for $1\leqslant i\leqslant {m'}_k$ with $|c_k^i|=k-2$ and such that $\psi_{k-1}(y_k^i-\lambda_k^iu_k^i)=dc_k^i$. Hence,
$$d_B(\psi_{k-1}(w_k^i)-c_k^i)=0,\; 1\leqslant i\leqslant {m'}_k.$$
Consequently, new non-zero classes   $[\psi_{k-1}(w_k^i)-c_k^i],\; \text{ for } 1\leqslant i\leqslant {m''}_k\leqslant {m'}_k$, are created  in $H^{k-2}(B,d_B)$, so that, following the same process we acquire $$W_k^{k-2}={w'}_k^1R\oplus \cdots\oplus {w'}_k^{{m''}_k}R,$$ with  $|{w'}_k^i|=k-2$ and $d{w'}_k^i=0$. We then   extend $\psi_{k-1}$ to what we  finally denote:
$$\varphi_k:\left( A\sqcup T(V_{\leqslant k-1}\oplus V_k^{k-1}\oplus V_k^{k-2}\oplus W_k^{k-1}\oplus W_k^{k-2}),d \right)\longrightarrow (B,d_B), $$
by putting $\varphi_k({w'}_k^i)=\psi_{k-1}(w_k^i)-c_k^i$. 

Notice that this extension resolves the issue of $H^{k-2}(\psi_{k-1})$ not being surjective. Therefore, we have $H^{\leqslant k}(\varphi_k)$ is an isomorphism.

Next, we write $H^{k+1}(B,d_B)=\operatorname{Im}(H^{k+1}(\varphi_k))\oplus \mathcal{B}_k$, $\mathcal{B}_k$ is finitely generated and let $[\beta_k^1]$, $\cdots$, $[\beta_k^{r_k}]$ be its basis such that $d\beta_k^i=0$. We introduce $$V_k^{k+1}=\alpha_k^1R\oplus\cdots\oplus \alpha_k^{r_k}R,$$ with  $|\alpha_k^i|=k+1$, and $d\alpha_k^i=0$. 

Finally, we put
\begin{equation*}\label{V-k}
V_k=V_k^{k-1}\oplus W_k^{k-1}\oplus V_k^{k-2}\oplus W_k^{k-2}\oplus V_k^{k+1}
\end{equation*}
 and extend, once again, $\varphi_k$ to
$$\psi_k:\left( A\sqcup T(V_{\leqslant k}),d \right)\longrightarrow (B,d_B) ,$$
such that $\psi_k(\alpha_k^i)=\beta_k^i$ for $1\leqslant i\leqslant r_k$.

Consequently, we have $H^{k+1}(\psi_k)$ is surjective, and $H^k(\psi_k)$ remains an isomorphism since there is no new cocycles of degree $k$. This finishes the $k$-th step and (since chomologies are assumed finitely generated)   by induction argument the proof of the proposition.
\end{proof}

If $(A,d)=(R, 0)$, we get  $\Phi: (TV,d)\stackrel{\simeq}{\rightarrow} (B,d)$ (or $(TV,d)$ for short) which we  call a {\it free  model} of $(B,d)$. If $R$ is a field, the differential $d$ is decomposable in the sense that $d: V\rightarrow T^{\geq 2}V$ (i.e. (\ref{indecomposability}) above does not exist). If $R$ is not a field, we still have a sort of minimality in the sense that, for each $i$,
 there is a non-invertible $r_i\in R$ such that $d: V^i \rightarrow r_iV^{i+1}$. In both cases,  $(TV,d)$ is said to be {\it minimal} and we call it a {\it minimal free model} of $(B,d)$. 

Notice that, if we put respectively:
$$
\left\{ \begin{array}{l}
V(0) =0,\; V(1) = V^2_2,\\
V(2) = V(1) \oplus V_3^2\oplus W_3^2 \oplus W_3^1,\\
V(3) = V(2)\oplus V_3^1\oplus V_3^4,\\
\vdots\qquad \qquad \vdots\\
V(k) = V(k-1) \oplus V_k^{k-1}\oplus W_k^{k-1} \oplus W_k^{k-2}\\
V(k+1) = V(k)\oplus V_k^{k-2}\oplus V_k^{k+1}.
\end{array}
\right.
$$
Clearly, this is an increasing sequence. Let then $V=\bigcup_iV(i)$.
We then  get $(A\sqcup T(V), d)$ as {\it free extension} in the sense \cite[Appendix]{HL}.
 The inclusion $\iota : (A,d) \hookrightarrow (A\sqcup TV,d)$ will be called a {\it free model} of $f$.
 Moreover, following the notation as before,  every element $z_{k+1}^i =du_{k+1}^i\in \left( T(V_{\leqslant k}) \right)^{k+1} $ cannot be expressed by a single element of degree $k+1$  (specially, since $V_1=0$), therefore we necessarily have $z_{k+1}^i\in T^{\geqslant 2}(V_{\leqslant k})$.
\begin{prop}\label{rel-free case}
Given a diagram of $R$-dga's
$$\begin{tikzcd}
 & (A,d_A)\arrow[d,"\eta","\simeq"']\\
 (TV,d)\arrow[r,"\psi"'] & (B,d_B)
\end{tikzcd}$$
such that:
\begin{itemize}
\item[(i)] $(TV,d)$ is a minimal dga.
\item[(ii)] $\eta$ is a surjective quasi-isomorphism.
\end{itemize}
Then there exists a morphism $\varphi :(TV,d)\rightarrow (A,d_A)$ such that $\eta\circ\varphi =\psi$.
\end{prop}
\begin{proof}
We proceed by induction  as follows:\\
Assume that $\varphi$ is constructed in $V(k-1)$, and let $v_\alpha$ be a basis of $V_k^{k-1}\oplus W_k^{k-1} \oplus W_k^{k-2}$. As $dv_\alpha\in A\sqcup T(V(k-1))$, $\varphi (dv_\alpha)$ is well-defined and we have $d_B\varphi(dv_\alpha)=\varphi (d^2v_\alpha)=0$.

By hypothesis, we have
$$\eta\circ\varphi(dv_\alpha)=\psi (dv_\alpha)=d_B\psi(v_\alpha),$$
and since $\eta$ is a surjective quasi-isomorphism, there exists $a_\alpha$ in $A$ such that $\eta (a_\alpha)=\psi (v_\alpha)$ and $d_Aa_\alpha =\varphi (dv_\alpha)$. We then extend $\varphi$ to $V_k^{k-1}\oplus W_k^{k-1} \oplus W_k^{k-2}$ by setting $\varphi (v_\alpha)=a_\alpha$, consequently we obtain a morphism $$\varphi :TV(k)\rightarrow (A,d_A)$$ satisfying $d_A\varphi (v_\alpha)=\varphi (dv_\alpha)$.
\\
We then do the same to extend $\varphi$ to $V_k^{k-2}\oplus V_k^{k+1}$ and therefore to obtain a morphism $$\varphi :TV(k+1)\rightarrow (A,d_A)$$ satisfying $d_A\varphi (v_\alpha)=\varphi (dv_\alpha)$.
\end{proof}

In the case of commutative differential graded algebras, following the same process as in the proof of Proposition \ref{free case} yields the commutative model of a morphism of cdga algebras.
\begin{prop}\label{commodel} Let $\left(A, d_A\right)$ and $\left(B, d_B\right)$ be two commutative differential graded algebras over $R$ satisfying:
\begin{itemize}
\item[(i)] $H^0\left(A, d_A\right)=H^0\left(B, d_B\right)=R$ and $H^1\left(A, d_A\right)=H^1\left(B, d_B\right)=0$.
\item[(ii)] $H^i\left(A, d_A\right)$ and $H^i\left(B, d_B\right)$ are finitely generated $R$-modules for all $i$.
\item[(iii)] $H^2\left(B, d_B\right)$ is a free $R$-module.
\item[(iv)] $A$ is $R$-free.
\end{itemize}
Let $f: (A,d_A) \rightarrow (B,d_B)$ be a morphism of $R$-cdga's such that
\begin{itemize}
\item[(v)] $H^2(f): H^2(A,d_A) \rightarrow H^2(B,d_B)$ is injective.
\end{itemize}
Then there exists a commutative diagram of $R$-cdga's of the form:
$$
\begin{tikzcd}
(A, d_A) \arrow[r,"f"]\arrow[rd,hook,"i"'] & (B, d_B)  \\
 & (A \otimes\Lambda V, d) \arrow[u,"\varphi"'] ,
\end{tikzcd}
$$
where:
\begin{itemize}
\item[1)] $V$ is a free $R$-module and $\Lambda V$ is the commutative algebra over $V$.
\item[2)] $i$ is the inclusion ($i(a)=a$).
\item[3)] $\varphi$ is a quasi-isomorphism of $R$-cdga's.
\end{itemize}
\end{prop}
\begin{proof}
We proceed by induction on $k\geqslant 0$, and we suppose that we have constructed $\psi_{k-1}:(A\otimes \Lambda (V_{\leqslant k-1}),d)\rightarrow (B,d_B)$ satisfying:
\begin{itemize}
\item[($a_{k-1}$)] $V=V_{\leqslant k-1}$ is a free $R$-module.
\item[($b_{k-1}$)] $H^{\leqslant k-1}(\psi_{k-1})$ is an isomorphism.
\item[($c_{k-1}$)] $H^k(\psi_{k-1})$ is surjective.
\end{itemize}

Notice that in this case, the free product and the tensor algebra become respectively the tensor product and the commutative algebra. Therefore, following the same process as in the  non-commutative case, we get $\psi_{k}:(A\otimes \Lambda (V_{\leqslant k}),d)\rightarrow (B,d_B)$ satisfying:
\begin{itemize}
\item[($a_{k}$)] $V=V_{\leqslant k}$ is a free $R$-module.
\item[($b_{k}$)] $H^{\leqslant k}(\psi_{k})$ is an isomorphism.
\item[($c_{k}$)] $H^{k+1}(\psi_{k})$ is surjective.
\end{itemize}
\end{proof}
Once again and similarly to the free case (cf. Remark following the proof of Proposition \ref{free case}),
if $(A,d)=(R, 0)$, we get a quasi-isomorphism $\Phi: (\Lambda V,d)\stackrel{\simeq}{\rightarrow} (B,d)$ whose source is a  cdga (compare with \cite[\S 7]{Ha}).
If $R$ is a field, the differential $d$ is decomposable in the sense that $d: V\rightarrow \Lambda ^{\geq 2}V$ (see(\ref{indecomposability}) below). If $R$ is not a field, we still have a sort of minimality in the sense of \cite[Theorem 7.1 (ii)]{Ha}, that is, for each $i$, there is a non-invertible $r_i\in R$ such that $d: V^i \rightarrow r_iV^{i+1}$.  In both cases,  $(\Lambda V,d)$ is a {\it minimal} cdga which  
we call  a {\it  commutative minimal model} of $(B,d)$. The inclusion $\iota : (A,d) \hookrightarrow (A\otimes \Lambda V,d)$ will be called a {\it minimal commutative  model} of $f$.

Following the proof of Proposition \ref{rel-free case}, we get (as in the previous case)
\begin{prop}\label{lift}
Given a diagram of $R$-cdga's
$$\begin{tikzcd}
 & (A,d_A)\arrow[d,"\eta","\simeq"']\\
 (\Lambda V,d)\arrow[r,"\psi"'] & (B,d_B)
\end{tikzcd}$$
such that:
\begin{itemize}
\item[(i)] $(\Lambda V,d)$ is a minimal cdga.
\item[(ii)] $\eta$ is a surjective quasi-isomorphism.
\end{itemize}
Then there exists a morphism $\varphi :(\Lambda V,d)\rightarrow (A,d_A)$ such that $\eta\circ\varphi =\psi$.
\end{prop}
\begin{cor}
\begin{enumerate}
\item Let $(A,d)$ be  a cdga  in  $\text{CDGA}_r(R)$ such that 
$H^*(A,d)$ satisfies  (\ref{finitegenfree}). Then,  there is, up to a quasi-isomorphism, a unique minimal commutative model for  $(A,d)$.
\item Let $f: (A,d)\rightarrow (B,d)$ be  a cdga   in  $CDGA_r(R)$ such that  $H^*(B,d)$ satisfies  (\ref{finitegenfree}).
Then, there is, up to a quasi-isomorphism, a unique minimal commutative model for  $f$.
\end{enumerate} 
\end{cor}
\begin{proof}
It suffice to use the above Proposition and \cite[Proposition 7.7 and Remark 7.8]{Ha} for two minimal models of $(A,d)$ (resp. $(B,d)$).
\end{proof}
\section{Sectional category}\label{secatComm}

In this section, we introduce the sectional category of a surjective morphism of $R$-cdga's, where $R$ is a principal ideal domain containing $\frac{1}{2}$.

Given an $(r,\rho(R'))$-H-mild algebra $(A,d)$, $(A\otimes A,d)$ can be viewed as an $(r,l)$-H-mild algebra. First recall the Künneth formula for cohomology
$$H^k(A\otimes A)=\bigoplus_{p+q=k} \left(  H^p(A)\otimes H^q(A) \right)  \oplus \bigoplus_{p'+q'=k+1}Tor_1(H^{p'}(A),H^{q'}(A)),$$
and notice that, when $k=2r\rho(R')$, the  term $\bigoplus_{p'+q'=k+1}Tor_1(H^{p'}(A),H^{q'}(A))$ vanishes  and $\bigoplus_{p+q=k} \left(  H^p(A)\otimes H^q(A) \right)$ reduces to two summand, so  that: $$H^{2r\rho(R')}(A\otimes A)=H^{r\rho(R')}(A)\otimes H^{r\rho(R')}(A).$$
Thus the cohomology, greater than $2r\rho(R')$, of $(A\otimes A,d)$ vanishes.  Consequently $(A\otimes A,d)$ is an $(r,2r\rho(R'))$-H-mild algebra. Following the same process we can conclude that, for any $m\geq 2$,  $(A^{\otimes m},d)$ is an $(r,mr\rho(R'))$-H-mild algebra.

Now let $\varphi : A\rightarrow B$ be a surjective morphism of $(r,\rho (R'))$-H-mild cdga's. We then change the principal ideal domain $R'$ to $R$ in a way for the quotient algebra $\left(\frac{{A}^{\otimes m+1}}{\left(\operatorname{ker} {{(\varphi )}}\right)^{\otimes m+1}}, \overline{d} \right)$ to be $(r,\rho(R))$-H-mild and such that $\rho(R)$ satisfies the following:
\begin{equation}\label{rR}
m\rho(R')\leq \rho(R).
\end{equation}
Consequently, the following projection
$$p_{m}:\left({A}^{\otimes m+1}, d\right) \longrightarrow\left(\frac{{A}^{\otimes m+1}}{\left(\operatorname{ker} {{(\varphi)}}\right)^{\otimes m+1}}, \overline{d} \right)$$
becomes a morphism of $(r,\rho(R))$-H-mild cdga's, and applying proposition \ref{commodel} yields the following commutative diagram
$$\begin{tikzcd}
\left(A^{\otimes m+1}, d\right) \arrow[r,"p_{m}"] \arrow[rd,hook,"i_m"']
&  \left(\frac{{A}^{\otimes m+1}}{\left(\operatorname{ker} {{(\varphi)}}\right)^{\otimes m+1}}, \overline{d} \right) \\
  &  (A^{\otimes m+1}\otimes \Lambda W_{(m)},d). \arrow[u,"\simeq"']
\end{tikzcd}$$
Thus, introducing  $\bar{\mu} _{m+1} : A^{\otimes m+1}\otimes \Lambda W_{(m)}\stackrel{\mu_{m+1}\otimes id_{\Lambda W_{(m)}}}{\longrightarrow} A\otimes \Lambda W_{(m)}$,
we then obtain the following diagram:
$$\begin{tikzcd}
\left(A^{\otimes m+1}, d\right) \arrow[r,"i_m",hook] \arrow[d,"\mu_{m+1}"'] & (A^{\otimes m+1}\otimes \Lambda W_{(m)},d) \arrow[d,"\bar{\mu}_{m+1}"] \\
(A,d) \arrow[r,"j_m"',hook]  & (A\otimes \Lambda W_{(m)},\bar{d}).
\end{tikzcd}$$
As already established for the rational case, the following pushout
\begin{equation}\label{secatchara}
\begin{tikzcd}
\left(A^{\otimes m+1}, d\right) \arrow[r,"i_m",hook] \arrow[d,"\mu_{m+1}"'] & (A^{\otimes m+1}\otimes \Lambda W_{(m)},d) \arrow[d,"\bar{\mu}_{m+1}"] \arrow[rdd,"\tau",bend left=15] & \\
(A,d) \arrow[r,"j_m"',hook] \arrow[rrd,"id_A"',bend right=15] & (A\otimes \Lambda W_{(m)},d) \arrow[rd,"r",dashed] & \\	
 & & (A,d)
\end{tikzcd}
\end{equation}
allows the following property: there exists a cdga morphism $\tau$ satisfying $\tau\circ i_m=\mu_{m+1}$ if and only if there is a cdga retraction $r$ for $j_m$. Therefore we have the following
\begin{prop}
Let $\varphi :A\rightarrow B$ be a surjective morphism of algebras, in $\text{CDGA}_r(R)$, whose homology satisfies (\ref{finitegenfree}).  Then, $secat(\varphi)$ (\ref{secat}) is the smallest $m$ for which $j_m$ admits a cdga retraction.

In particular, if $X$ is a finite type $(r,\rho(R))$-mild CW-complex such that $H_*(\Omega X, R)$ is torsion free, and $(\Lambda V,d)\longrightarrow C^*(L)$ is a decomposable minimal commutative model of its  Cartan-Eilenberg-Chevally complex $C^*(L)$, then $TC_n(X,R)=secat(\mu_{\Lambda V,n})$.
\end{prop}
This involves the following
\begin{definition}\label{secat}
Let $\varphi :A\rightarrow B$ be a surjective morphism of algebras, in $\text{CDGA}_r(R)$, whose homology satisfies (\ref{finitegenfree}).
\begin{itemize} 
\item[(i)] The module sectional category, $msecat(\varphi)$, of $\varphi$ is the smallest $m$ such that $j_m$
admits an $A$-module retraction.
\item[(ii)] The homology sectional category, $Hsecat(\varphi)$, of $\varphi$ is the smallest $m$ such that $H(j_m)$ is injective.
\end{itemize}
\end{definition}

The inequality $$nil \ker{(H(\varphi))} \leq Hsecat(\varphi),$$ still holds exactly as in \cite[Proposition 6]{JCar}. Here, $nil \ker{(H(\varphi))}$ is the longest non trivial product of elements of $\ker{(H(\varphi))}$.
Consequently we have (in our context):
 $$nil \ker{(H(\varphi))} \leq Hsecat(\varphi) \leq msecat(\varphi) \leq secat(\varphi).$$

 Similarly, if  $X$ satisfies the hypothesis of the above proposition, we get 
 the following
\begin{definition}
\begin{enumerate}
 \item[(i)] $mTC_n(X,R)$ is defined as $msecat(\mu_{\Lambda V,n})$.
 \item[(ii)] $HTC_n(X,R)$ is defined as $Hsecat(\mu_{\Lambda V,n})$.
 \end{enumerate}
\end{definition}
This translates the above inequalities to topological ones as follows:
$$nil \ker(H(\mu_{\Lambda V,n}))\leq HTC_n(X,R)\leq mTC_n(X,R)\leq TC_n(X, R)$$
where $nil \ker(H(\mu_{\Lambda V,n}))$ is the longest non trivial product of elements of $\ker(H(\mu_{\Lambda V,n}))$.

\section{$sc$ invariant and topological complexity}

In this section, inspired by the study of the (non-homotopy) invariant $\operatorname{sc}$ in \cite{JCar}, we construct an analogous definition for a surjective morphism of $(r,\rho(R))$-H-mild cdga's in terms of its commutative model over a principal ideal domain $R$.

Let $\varphi :(A,d)\rightarrow (B,d)$ be a surjective morphism of $(r,\rho(R'))$-H-mild cdga's, we transform the principal ideal domain to $R$ satisfying (\ref{rR}), that is, such that the following projection
$$\Gamma_{m} :(A,d)\rightarrow \left( \frac{A}{(\ker (\varphi))^{m+1}},\overline{d}\right) $$
becomes a morphism of $(r,\rho(R))$-H-mild cdga's. Therefore, applying Proposition \ref{commodel}, we have for each integer $m$ a commutative model of $\Gamma_{m}$:
$$\begin{tikzcd}
\left(A, d\right) \arrow[r,"\Gamma_{m}"] \arrow[rd,hook,"\iota_m"']
&  \left(\frac{A}{\left(\operatorname{ker} {{(\varphi)}}\right)^{m+1}}, \overline{d} \right) \\
  &  (A\otimes \Lambda V_{(m)},d), \arrow[u,"\simeq"']
\end{tikzcd}$$
allowing the following
\begin{definition}
Let $\varphi :(A,d)\rightarrow (B,d)$ be a surjective morphism of algebras, in $\text{CDGA}_r(R)$, whose homology satisfies (\ref{finitegenfree}) and consider the previous diagram:
\begin{itemize}
\item[(i)] $sc(\varphi)$ is the smallest integer $m$ such that $\Gamma_{m}$ admits a homotopy retraction.
\item[(ii)] $msc(\varphi)$ is the smallest $m$ such that $\Gamma_{m}$ admits a homotopy retraction as $A$-modules.
\item[(iii)] $Hsc(\varphi)$ is the smallest $m$ such that $H(\Gamma_{m})$ is injective.
\end{itemize}
\end{definition}

\begin{theo}\label{upbound}
Let $\varphi$ be any surjective morphism of $(r,\rho(R))$-H-mild whose homology satisfies (\ref{finitegenfree}), we have
\begin{itemize}
\item[(i)] $secat(\varphi)\leq sc(\varphi)$,
\item[(ii)] $msecat(\varphi)\leq msc(\varphi)$,
\item[(iii)] $Hsecat(\varphi)\leq Hsc(\varphi)$.
\end{itemize}
\end{theo}

\begin{proof}
Let $(A\otimes \Lambda V_{(m)},d) \stackrel{\simeq}{\longrightarrow} \left( \frac{A}{\ker{(\varphi)}} ,\overline{d} \right) $ be  a commutative model for $\Gamma_m$. The multiplication map $\mu_{m+1}:A^{\otimes m+1}\rightarrow A$ induces the following map
$$\overline{\mu}_{m+1}: \frac{A^{\otimes m+1}}{({\ker(\varphi)})^{\otimes m+1}}\rightarrow \frac{A}{(\ker(\varphi))^{m+1}} ,$$
fitting in the following commutative diagram
$$\begin{tikzcd}
A^{\otimes m+1}\arrow[r,"\mu_{m+1}"]\arrow[d,"i_m"',hook] & A \arrow[r,"\iota_m",hook] & A\otimes \Lambda V_{(m)} \arrow[d,"\simeq"]\\
A^{\otimes m+1}\otimes \Lambda W_{(m)}\arrow[r,"\simeq"']\arrow[rru,"\theta",dashed] & \frac{A^{\otimes m+1}}{({\ker(\varphi)})^{\otimes m+1}}\arrow[r,"\overline{\mu}_{m+1}"'] & \frac{A}{(\ker(\varphi))^{m+1}},
\end{tikzcd}$$
where $\theta$ is obtained by the lifting lemma (Proposition \ref{lift}).

Now if $r:(A\otimes \Lambda V_{(m)},d)\rightarrow (A,d)$ is a homotopy retraction for $\Gamma_m$, then put $\tau=r\circ\theta$ and clearly we have $\tau\circ i_m=\mu_{m+1}$, and this is equivalent to $j_m$ having an algebra retraction (\ref{secatchara}).
\end{proof}

\begin{cor}
Let $\varphi:A\rightarrow B$ be a surjective morphism of algebras, in $\text{CDGA}_r(R)$, whose homology satisfies (\ref{finitegenfree}), then
$$secat(\varphi) \leq nil \ker(\varphi).$$
\end{cor}
As an application of this definition, we introduce, for $n\geq 2$, algebraic topological complexities $\operatorname{tc_n}$, $\operatorname{mtc_n}$ and $\operatorname{Htc_n}$ over a principal ideal domain. For this purpose, we consider $X$ to be a finite type $(r,\rho(R))$-mild CW-complex such that $H_*(\Omega X, R)$ is torsion free and let 
$$(\Lambda V,d)\longrightarrow C^*(L),$$
be its minimal commutative model.

Let $\mu_{\Lambda V,n} : (\Lambda V)^{\otimes n}\rightarrow \Lambda V$ denote the $n$-fold product of $\Lambda V$,
same as before, making the necessary assumption on $\rho(R)$,
the commutative diagram follows:
$$\begin{tikzcd}
\left((\Lambda V)^{\otimes n}, d\right) \arrow[r,"\Gamma_{m }"] \arrow[rd,hook,"\iota_m"']
&  \left(\frac{(\Lambda V)^{\otimes n}}{\left(\operatorname{ker} {{(\mu_{\Lambda V,n})}}\right)^{m+1}}, \overline{d} \right)  \\
  &  ((\Lambda V)^{\otimes n}\otimes \Lambda V_{(m)},d), \arrow[u,"\simeq"']
\end{tikzcd}$$
and we have the following
 \begin{definition}\label{TC-R}
 \begin{enumerate}
 \item[(i)] $tc_n(X,R)$ is defined as $sc(\mu_{\Lambda V,n})$.
 \item[(ii)] $mtc_n(X,R)$ is defined as $msc(\mu_{\Lambda V,n})$.
 \item[(iii)] $Htc_n(X,R)$ is defined as $Hsc(\mu_{\Lambda V,n})$.
 \end{enumerate}
 \end{definition}
 
Another corollary of theorem \ref{upbound} is the following

\begin{cor}
Let $X$ be a CW-complex satisfying the above hypotheses, we have
\begin{itemize}
\item[(i)] $TC_n(X,R)\leq tc_n(X,R)$,
\item[(ii)] $mTC_n(X,R)\leq mtc_n(X,R)$,
\item[(iii)] $HTC_n(X,R)\leq Htc_n(X,R)$.
\end{itemize}
\end{cor}

\section{Homology nilpotency}
Let $I$ be an ideal of a cdga $A$, the homology nilpotency of $I$ is
$$Hnil I:=min \left\lbrace k: I^{k+1}\subset J, J \text{ is an acyclic ideal of } A \right\rbrace .$$
Let $\varphi:A\rightarrow B$ be a surjective morphism of algebras, in $\text{CDGA}_r(R)$, whose homology satisfies (\ref{finitegenfree}), and suppose that $(\ker(\varphi))^{m+1}$ is included in some acyclic ideal $J$ of $A$, which allows the commutative diagram
$$\begin{tikzcd}
A\arrow[r,"\Gamma_m"]\arrow[rd,"\simeq"'] & \frac{A}{(\ker(\varphi))^{m+1}}\arrow[d]\\
 & A/J,
\end{tikzcd}$$
we combine it with the commutative model yields the morphism $r$ via the lifting lemma (Proposition \ref{lift})
$$\begin{tikzcd}
 & & A\arrow[d,"\simeq"]\\
A\otimes\Lambda W_{(m)}\arrow[r,"\simeq"']\arrow[rru,"r",dashed] & \frac{A}{(\ker(\varphi))^{m+1}} \arrow[r] & A/J
\end{tikzcd}$$
and we consequently have
$$\begin{tikzcd}
A\arrow[d,"id_A"',equal]  &  A\otimes\Lambda W_{(m)}\arrow[d,"\simeq"]\arrow[l,"r"']  \\
A\arrow[ru,"\iota_m",hook] \arrow[rd,"\simeq"']\arrow[r,"\Gamma_m"']  &  \frac{A}{(\ker(\varphi))^{m+1}} \arrow[d] \\
  &  A/J,
\end{tikzcd}$$
 which commutes, therefore we have $r\circ\iota_m=id_A$, thus
 $$sc(\varphi)\leq Hnil \ker(\varphi).$$
 \begin{prop}
 Let $\varphi:A\rightarrow B$ be a surjective morphism of algebras, in $\text{CDGA}_r(R)$, whose homology satisfies (\ref{finitegenfree}). Then,
 $$ nil \ker(H(\varphi))\leq secat(\varphi)\leq sc(\varphi) \leq Hnil \ker(\varphi) \leq nil \ker(\varphi) .$$
 \end{prop}

 \begin{cor} Let $X$ be a finite type $(r,\rho(R))$-mild CW-complex such that $H_*(\Omega X, R)$ is torsion free. Then,
 $$ nil \ker(H(\mu_{\Lambda V,n}))\leq TC_n(X,R)\leq tc_n(X,R)\leq Hnil \ker(\mu_{\Lambda V,n})\leq nil \ker(\mu_{\Lambda V,n}).$$
 \end{cor}

\section{Asecat and Asc}

In this section, we introduce $Asecat$ and $Asc$ of a surjective dga's, respective analogs of $secat$ and $sc$ in the commutative case.  Let $\varphi :(A,d)\rightarrow (B,d)$ be a surjective morphism of dga's whose homology satisfies (\ref{finitegenfree}), and consider the projection
$$p_{m}:\left({A}^{\otimes m+1}, d\right) \longrightarrow\left(\frac{{A}^{\otimes m+1}}{\left(\operatorname{ker} {{(\varphi)}}\right)^{\otimes m+1}}, \overline{d} \right).$$
By proposition \ref{free case}, $p_m$ admits a free model $i_m$
$$\begin{tikzcd}
\left(A^{\otimes m+1}, d\right) \arrow[r,"p_{m}"] \arrow[rd,hook,"i_m"']
&  \left(\frac{{A}^{\otimes m+1}}{\left(\operatorname{ker} {{(\varphi)}}\right)^{\otimes m+1}}, \overline{d} \right) \\
  &  (A^{\otimes m+1}\sqcup T W_{(m)},d). \arrow[u,"\simeq"']
\end{tikzcd}$$
Same as in section \ref{secatComm}, the pushout
\begin{equation*}
\begin{tikzcd}
\left(A^{\otimes m+1}, d\right) \arrow[r,"i_m",hook] \arrow[d,"\mu_{m+1}"'] & (A^{\otimes m+1}\sqcup T W_{(m)},d) \arrow[d,"\bar{\mu}_{m+1}"] \arrow[rdd,"\tau",bend left=15] & \\
(A,d) \arrow[r,"j_m"',hook] \arrow[rrd,"id_A"',bend right=15] & (A\sqcup T W_{(m)},d) \arrow[rd,"r",dashed] & \\	
 & & (A,d)
\end{tikzcd}
\end{equation*}
allows the following
\begin{definition}
For any surjective morphism $\varphi :(A,d)\rightarrow (B,d)$ of dga's whose homology satisfies (\ref{finitegenfree}), $Asecat(\varphi)$ is the least $m$ such that there exists a morphism of dga's 
$$\tau : (A^{\otimes m+1}\sqcup T W_{(m)},d) \rightarrow (A,d)$$ satisfying $\tau \circ i_m=\mu_{m+1}$. Equivalently, $m$ the least integer such that there is a dga retraction $r$ of $j_m$.
\end{definition}
Now consider a $1$-connected $R$-finite type space $X$, that is, $H^0(X;R)=R$, $H^1(X;R)=0$ and $H^i(X;R)$ is finite type $R$-module for all $i$. It follows from the universal coefficients theorem that $H^2(X;R)$ is a free $R$-module. Therefore $C^*(X;R)$ satisfies (\ref{finitegenfree}) (for $r=1$), consequently $C^*(X;R)$ admits a free model
$$(TV,d)\stackrel{\simeq}{\longrightarrow} C^*(X;R),$$
\textit{$(TV,d)$ is called a free minimal $R$-model of $X$.} It results the following
\begin{cor}
Every $1$-connected $R$-finite type space admits a free minimal $R$-model, and the model is decomposable if and only if $H_*(\Omega X;R)$ is $R$-free.
\end{cor}
\begin{definition}
Let $X$ be a $1$-connected $R$-finite type space that admits a free decomposable model $(TV,d)$, then
$$ATC_n(X,R)=Asecat(\mu_{TV,n}),$$
where $\mu_{TV,n}:{TV}^{\otimes n}\rightarrow TV$ is the $n$-fold multiplication on $TV$.
\end{definition}
Similarly, we consider a surjective morphism $\varphi :(A,d)\rightarrow (B,d)$ of dga's whose homology satisfies (\ref{finitegenfree}), applying proposition \ref{free case} to the projection
$$\Gamma_{m} :(A,d)\rightarrow \left( \frac{A}{(\ker (\varphi))^{m+1}},\overline{d}\right) $$
we acquire a free model of $\Gamma_{m}$
$$\begin{tikzcd}
\left(A, d\right) \arrow[r,"\Gamma_{m}"] \arrow[rd,hook,"\iota_m"']
&  \left(\frac{A}{\left(\operatorname{ker} {{(\varphi)}}\right)^{m+1}}, \overline{d} \right) \\
  &  (A\sqcup T V_{(m)},d), \arrow[u,"\simeq"']
\end{tikzcd}$$
and we consequently have the following
\begin{definition}
With the notation and the diagram above, $Asc(\varphi)$ is the least integer $m$ such that $\Gamma_{m}$ admits a homotopy retraction.
\end{definition}
As an application of the latter construction, we have the following
\begin{definition}
Let $X$ be a $1$-connected $R$-finite type space admitting a free decomposable model $(TV,d)$, then
$$Atc_n(X,R)=Asc(\mu_{TV,n}).$$
\end{definition}

The same argument as in the commutative case allows the following
\begin{prop}
Let $\varphi :(A,d)\rightarrow (B,d)$ be a surjective morphism satisfying \ref{finitegenfree}, then
$$Asecat (\varphi)\leq Asc (\varphi). $$
\end{prop}
As a result we have
\begin{cor}
For every $1$-connected $R$-finite type space $X$ admitting a free decomposable model $(TV,d)$ we have
$$ATC(X,R)\leq Atc(X,R).$$
\end{cor}

Next, we combine the two invariants from the context of dga's and the context of commutative dga's, respectively through free models and commutative models of a specific $(r,\rho(R))$-mild CW-complex, and extract some inequalities between them. For this purpose, we let $R$ designate a principal domain containing $\frac{1}{2}$.
\begin{prop}
For every finite type $(r,\rho(R))$-mild CW-complex $X$ such that $H_*(\Omega X;R)$ is torsion free, we have 
$$\begin{tabular}{ccc}
$ATC_n(X,R)$ & $\leq$ & $TC_n(X,R)$, \\ 
$Atc_n(X,R)$ & $\leq$ & $tc_n(X,R)$. \\ 
\end{tabular} $$
\end{prop}
\begin{proof}
On one hand, $X$ admits a free decomposable model
$$(TV,d)\stackrel{\simeq}{\longrightarrow} C^*(X;R).$$
On the other hand, $X$ admits a commutative model
$$(\Lambda W,d)\stackrel{\simeq}{\longrightarrow} C^*(L),$$
combining the two models yields the following
$$\begin{tikzcd}
C^*(X;R) & B(C_*(\Omega (X),R)^{\vee}\arrow[l,"\simeq"']\arrow[r,"\simeq"] & C^*(L)  \\
TV \arrow[u,"\simeq"] \arrow[rr] && \Lambda W. \arrow[u,"\simeq"']
\end{tikzcd}$$
The rest of the proof is analogous th that of \cite[Theorem 3.3 (i)]{HL}.
\end{proof}

\end{document}